\documentclass[12pt]{amsart}
\usepackage{a4wide,amsmath,amsthm,amsfonts,amssymb}
\newtheorem{theorem}{Theorem}
\newtheorem{lemma}{Lemma}
\newtheorem{definition}{Definition}\theoremstyle{definition}

\begin{document}
\author{Mytrofanov M. A., Ravsky A.V.}
\title{A note on approximation of\\ continuous functions on normed spaces}


\begin{abstract}
Let $X$ be a real separable normed space $X$ admitting a separating polynomial. We prove that each
continuous function from a subset $A$ of $X$ to a real Banach space can be uniformly approximated
by restrictions to $A$ of functions which are analytic on open subsets of $X$. Also we prove that
each continuous function to a complex Banach space from a complex separable normed space admitting
a separating $*$-polynomial can be uniformly approximated by $*$-analytic functions.

\end{abstract}


\address{Pidstryhach Institute for Applied Problems of Mechanics and Mathematics
National Academy of Sciences of Ukraine, Lviv, Ukraine}

\email{mishmit@rambler.ru (Mytrofanov M. A.)}
\email{alexander.ravsky@uni-wuerzburg.de (Ravsky A.V.)}

\maketitle


The first known result on uniform approximation of
continuous functions was obtained by Weierstrass in 1885.
Namely, he showed that any continuous real-valued function on a compact
subset $K$ of a finitely dimensional real Euclidean space $X$ can be uniformly approximated by
restrictions on $K$ of polynomials on $X$. For a compact subset $K$ of a finitely dimensional
complex Euclidean space $X$ holds a counterpart of Stone-Weierstrass' theorem, according to
which any continuous complex-valued function on $K$ can be approximated by elements of any
algebra containing restrictions on $K$ of polynomials on $X$ and their conjugated functions.
A general direction of investigations
is to try to extend these results to topological linear spaces.
Most of the obtained results concern separable Banach spaces, although
in the paper~\cite{Mit_Rav} the authors obtained partial positive results for separable Fr\'echet spaces.
A negative result belong to Nemirovskii and Semenov, who in \cite{NS} built a
continuous real-valued function on the unit ball $K$ of the real space $\ell_2$
which cannot be uniformly approximated by restrictions onto $K$ of polynomials on $\ell_2$.
This result showed that in order to uniformly approximate continuous functions on Banach spaces we
need a bigger class of functions than polynomials.
The following fundamental result was obtained by Kurzweil~\cite{Ku}.

\begin{theorem}\label{theorem2}
Let $X$ be any separable real Banach space that admits a separating polynomial,
$G$ be any open subset of $X$, and $F$ be any continuous map
from $G$ to any real Banach space $Y$. Then
for any $\varepsilon > 0$ there exists an analytic map $H$ from $G$ to $Y$
such that $\|F(x)-H(x)\|<\varepsilon$ for all $x \in G$.
\end{theorem}

Separating polynomials were introduced in~\cite{Ku} and are considered in
reviews ~\cite{GRJJ_1} and~\cite{M_33}.
In order to define them and to obtain a counterpart of Kurzweil's Theorem for
a complex Banach space $X$, in paper~\cite{Mitr} were introduced notions, which we
adapt below for complex normed spaces $X$ and $Y$.

A map $B_{km}$ from $X^{k+m}$ to $Y$ is a map of type $(k,m)$ if
$B_{km}(x_1,...,x_k,x_{k+1},...,x_{k+m})$ is a nonzero map which is
$k$-linear with respect to $x_i$, if $1 \le i \le k$ and $m$-antilinear
with respect to $x_{k+j}$, if $1 \le j \le m$.

\begin{definition}\label{1d1}
A map $B_n:X^{n} \to Y$ is $*$-$n$-linear if
$$B_n(x_1,...,x_k,x_{k+1},...,x_{k+m}) = \sum_{k+m=n}
c_{km}B_{km}(x_1,...,x_k,x_{k+1},...,x_{k+m}),$$ where
for each $k$ and $m$ such that $k+m=n$,
$B_{km}$  is a map of type $(k,m)$ and $c_{km}$ is either $0$ or $1$,
and at least one of $c_{km}$ is non-zero.
\end{definition}

\begin{definition}\label{1d4}
A map $F_n:X \to Y$ is called a $n$-homogeneous
$*$-polynomial if there exists a $*$-$n$-linear map
$B_n:X^n \to Y$ such that $F_n(x)=B_n(x,...,x)$ for all $x \in
X$. Remark that $F_0$ is a constant map.
\end{definition}

\begin{definition}\label{1d5}
A map $F:X\to Y$ is a $*$-polynomial of degree $j$, if $ F=\sum\limits_{n=0}^j F_n,$ where
$F_n$ is an $n$-homogeneous continuous $*$-polynomial
for each $n$ and $F_j \ne 0.$
\end{definition}

\begin{definition}\label{1d9}
A map $H:X \to Y$ is $*$-analytic if every point $x \in X$
has a neighborhood $V$  such that $H(x)=\sum\limits_{n=0}^{\infty} F_n(x),$ where
for each $n$ we have that $F_n$ is an $n$-homogeneous continuous $*$-polynomial
and the series $\sum\limits_{n=0}^{\infty} F_n(x)$ converges in $V$
uniformly with respect to the norm of the space $Y.$
\end{definition}

\begin{definition}\label{1d8} Let $X$ be a complex (resp. real) normed space.
A $*$-polynomial (resp. po\-ly\-no\-mi\-al) $P:X \to \mathbb{C}$ (resp. to $\mathbb R$) is called a separating $*$-polynomial (resp. polynomial) if $P(0)=0$ and $\inf\limits_{\|x\|=1} P(x)>0$.
\end{definition}

Denote by $\tilde {\mathcal H}(X,Y)$ the normed space of $*$-analytic functions from $X$ to $Y$.

\begin{theorem}\label{1t3}
Let $X$ be any separable complex Banach space that admits a separating $*$-po\-lynomial,
$Y$ be any complex Banach space, and $F:X\to Y$ be any continuous map.
Then for any $\varepsilon>0$ there exists a map $H\in \tilde {\mathcal H}(X,Y)$ such that
$\|F(x)-H(x)\|<\varepsilon$ for all $x \in X$.
\end{theorem}

The aim of the present paper is to generalize Theorems~\ref{theorem2} and~\ref{1t3} to normed spaces.
For this we need the following technical

\begin{lemma}\label{lem:SepCompl} If a real normed space $X$ admits a separating polynomial $q$ then
its completion $\hat X$ admits a separating polynomial too.
\end{lemma}

\begin{proof}
 We have $q=\sum_{i\in I} q_i$ is a sum of homogeneous polynomials $q_i$ on the space $X$. For each $i\in I$ there exists a polylinear form $h_i:X^{n_i}\to\mathbb{R}$ such that $q_i(x)=h_i(x,\dots,x)$ for each $x\in X$.
Since $h_i$ is a Lipschitz function on $X^{n_i}$, by~\cite[Theorem 4.3.17]{En}, it admits a
continuous extension $\hat h_i$ on the space $\hat X^{n_i}$, which is polylinear by the
polylinearity of $h_i$.
 The map $\hat q_i:\hat X\to\mathbb R$ defined as
$\hat q_i(x)=\hat h_i(x,\dots,x)$ for each $x\in\hat X$ is an extension of the map $q_i$.
Then the map $\hat q=\sum_{i\in I} \hat q_i$ is a continuous polynomial extension of
the map $q$ onto the space $X$. It is easy to show that the unit sphere $S$ of the space $X$ is dense
in the unit sphere $\hat S$ of the space $\hat X$. Therefore $\inf_{x\in \hat S} \hat q(x)=
\inf_{x\in S} q(x)>0$, so $\hat q$ is a separating polynomial for the space $\hat X$.
\end{proof}

\begin{theorem}\label{thm:main} Let $X$ be a separable real normed space
that admits a separating polynomial, $Y$ be a real Banach space,
$A \subset X$, $f: A \to Y$ be a continuous function, and $\varepsilon >0$.
Then there are an open set $A_\varepsilon \supset A$ of $X$
and an analytic function  $f_\varepsilon:A_\varepsilon\to Y$ such that
$\| f (x)-f_\varepsilon (x)\| <\varepsilon $ for all $ x \in A $.
\end{theorem}

\begin{proof}
Let $\hat X$ be a completion of $X$.
We build a cover of the set $A$ by open in $\hat X$ sets as follows.
For each point $x \in A$ pick its neighborhood $O(x)$ open in $\hat X$
such that\linebreak
 $\|f(x^{\prime})-f(x)\|<\varepsilon/3$ for all $x^{\prime} \in O(x)\cap A.$

Put $\hat A_\varepsilon=\bigcup\limits_{x\in A}O(x).$ The topological space
$\hat A_\varepsilon$ is metrizable, and therefore paracompact,~\cite[5.1.3]{En}.
Therefore, by~\cite[5.1.9]{En} there is a locally finite partition $\{\varphi_s:s \in S \}$
of the unity, subordinated to the cover $\{ O(x) : x \in A \}.$

Now we construct an auxiliary function $f'_{\varepsilon} : \hat A_{\varepsilon} \rightarrow Y$.
First, for each index $s \in S$ we define a real number $a_s$ as follows.
If $\operatorname{supp} \varphi_s \cap A  \neq \oslash$,
then we pick an arbitrary point $x_s \in \operatorname{supp} \varphi_s \cap A$, and we put $a_s=f(x_s)$. Otherwise, we put $a_s=0$. Finally, put $f'_{\varepsilon} = \sum_{s\in S} a_s \varphi_s$.

Let $x \in A$. Put $S_x = \{ s \in S : x \in \operatorname{supp} \varphi_s \}.$
Then $\sum_{ s \in S_x} \varphi_s(x)=1.$ Let $s \in S_x$ be any index.
Thus there is an element $x_0 \in A$ such that $x \in \operatorname{supp} \varphi_s\subset O(x_0).$
Hence $x_s \in O(x_0)$ and $\|f(x) - a_s\|=\|f(x) - f(x_s)\|\leqslant \|f(x)- f(x_0)\| +
\|f(x_0) - f(x_s)\| < 2\varepsilon/3$.
Then
$$\left\|f(x)- f'_\varepsilon(x)\right\|=\left\|f(x) - \sum_{ s \in S} a_s \varphi_s(x)\right\|=
\left\|\sum_{ s \in S}f(x)\varphi_s(x)-\sum_{ s \in S} a_s \varphi_s(x)\right\|=$$
$$\left\|\sum_{ s \in S_x}f(x)\varphi_s(x)-\sum_{ s \in S_x} a_s \varphi_s(x)\right\| \leqslant
\sum_{ s \in S_x}\left\|f(x)\varphi_s(x)-a_s \varphi_s(x)\right\|=$$
$$\sum_{ s \in S_x}\left\|f(x)-a_s\right\|\varphi_s(x) < \sum_{ s \in S_x} ( 2 \varepsilon  /3 )\varphi_s(x)=2\varepsilon/3.$$

The function $f'_\varepsilon$ is continuous on $\hat A_\varepsilon$ as a sum of a family of continuous functions with a locally finite family of supports.

By Lemma~\ref{lem:SepCompl}, the space $\hat X$ admits a separating polynomial.
Therefore the space $X$ satisfies the conditions of
Theorem~\ref{theorem2}, so there exists a function $\hat f_\varepsilon $ analytic on
$\hat A_\varepsilon$ such that $\|\hat f_\varepsilon(x) - f'_\varepsilon(x)\|< \varepsilon/3$
for all $x \in \hat A_\varepsilon.$ Then for all $x \in A$,
$\|f(x) - \hat f_\varepsilon(x)\| \leqslant
\|f(x) - f'_\varepsilon(x)\| + \|f'_\varepsilon(x) - \hat f_\varepsilon(x)\|<
\varepsilon$. It remains to put $A_\varepsilon=\hat A_\varepsilon\cap X$ and let
$f_\varepsilon$ be the restriction of the map $\hat f_\varepsilon$ to the set $A_\varepsilon$.
\end{proof}

For a complex normed space $X$ by $\tilde X$ we denote it considered as a real normed space,
and by ${\mathcal H}(\tilde X,Y)$ the real normed space  of analytic functions from
the space $\tilde X$ to a Banach space $Y$.

\begin{theorem}\label{1t5}
Let $X$ be any separable complex normed space that admits a separating $*$-polynomial,
$Y$ be any complex Banach space, and $F:X\to Y$ be any continuous map.
Then for any $\varepsilon>0$ there exists a map $H\in \tilde {\mathcal H}(X,Y)$ such that
$\|F(x)-H(x)\|<\varepsilon$ for each $x \in X$.
\end{theorem}
\begin{proof}
The proof is almost identical to the proof of Theorem 4 from~\cite{Mitr} with
the following modifications. Instead of the application of Kurzweil's Theorem we
apply Theorem~\ref{thm:main}. Instead of~\cite[Lemma 2]{Mitr} we use
the fact (whose proof is similar to that of~\cite[Lemma 2]{Mitr}) that
the identity map from a complex normed space $\tilde {\mathcal H}(X,Y)$
to the real normed space ${\mathcal H}(\tilde X,Y)$ is an isomorphism of real normed spaces.
\end{proof}

\end{document}